\documentclass[11pt]{amsart}
\usepackage{amssymb,amsthm,amsmath,amsfonts}

\usepackage{geometry}
\usepackage{verbatim}
\usepackage[utf8]{inputenc}
\usepackage{enumerate}
\usepackage{hyperref}

\usepackage{tikz}
\usetikzlibrary{positioning}

\newcommand{\mb}[1]{\mathbb{{#1}}}
\newcommand{\mc}[1]{\mathcal{{#1}}}

\newcommand{\dd}{\mathrm{d}}
\newcommand{\1}{\mathbf{1}}
\DeclareMathOperator{\tr}{tr}

\DeclareMathOperator{\supp}{supp}
\newcommand{\R}{\mathbb{R}}

\newcommand{\E}{\mathbb{E}}
\DeclareMathOperator{\var}{Var}
\newcommand{\abs}[1]{\left\vert #1 \right\vert}
\newcommand{\scal}[2]{\left\langle #1, #2\right\rangle}
\theoremstyle{definition}

\newtheorem{prob}{Problem}
\newtheorem{cor}{Corollary}
 
\newtheorem{thm}{Theorem}
\newtheorem{lemma}{Lemma}

\theoremstyle{definition}
\newtheorem{definition}{Definition}

\theoremstyle{remark}
\newtheorem{rem}{Remark}
\newtheorem{question}{Question} 

\begin{document}

\newgeometry{tmargin=2.5cm, bmargin=2.5cm, lmargin=2.5cm, rmargin=2.5cm}

\title[Variance-entropy comparison for Gaussian quadratic forms  ]{Sharp variance-entropy comparison for nonnegative Gaussian quadratic forms}

\author{Maciej Bartczak}
\address{University of Warsaw}
\email{mb384493@students.mimuw.edu.pl}

\author{Piotr Nayar}
\thanks{M.B. was supported by the National Science 
Centre, Poland, grant 2015/18/A/ST1/00553. P.N. and S.Z. were supported by the National Science Centre, Poland, grant 2018/31/D/ST1/01355}
\address{University of Warsaw}
\email{nayar@mimuw.edu.pl}

\author{Szymon Zwara}
\address{University of Warsaw}
\email{szymon.zwara@students.mimuw.edu.pl}

\begin{abstract}
In this article we study weighted sums of $n$ i.i.d.  Gamma($\alpha$) random variables with nonnegative weights. We show that for $n \geq 1/\alpha$ the sum with equal coefficients maximizes  differential entropy when variance is fixed. As a consequence, we prove that among nonnegative quadratic forms in $n$ independent standard Gaussian random variables, a diagonal form with equal coefficients maximizes differential entropy, under a fixed variance. 
This provides a sharp lower bound for the relative entropy between a nonnegative quadratic form and a Gaussian random variable. Bounds on capacities of  transmission channels subject to $n$ independent additive gamma noises are also derived. 
%We also prove that differential entropy of a weighted sum of i.i.d. exponential random variables with nonnegative weights is maximized, under fixed variance, when the weights are equal.   
\end{abstract}

\maketitle

{\footnotesize

}
\bigskip

\section{Introduction}\label{sec:intro}

For a random variable $X$ with density $f$ its Shannon differential entropy is defined by the formula $h(X)=-\int f \ln f$, provided that this integral converges, with the convention that $0 \ln 0 = 0$. It is a classical fact that $h(X) \leq h(G)$ if $X$ is a random variable with finite second moment and $G$ is a Gaussian random variable satisfying $\var(X)=\var(G)$. Thus, a Gaussian random variable maximizes entropy under a fixed variance (note that even if $X$ has finite second moment, the integral in the definition of $h(X)$ may diverge to $-\infty$, but never to $+\infty$). This statement can be rewritten in the form of a variance-entropy comparison as follows: for any random variable $X$ with finite second moment one has $h(X) \leq \frac12 \ln\left( 2\pi e \var(X) \right)$, see e.g. Theorem 8.6.5 in \cite{CT06}. Due to the Pinsker-Csisz\'ar-Kullback inequality, see \cite{P64,C62,K67}, one has $d_{\textrm{TV}}(X, G) \leq 2(h(G)-h(X))$, whenever $G$ is a Gaussian random variable with the same mean and variance as the random variable $X$. Here $d_{\textrm{TV}}$ stands for the total variation distance. Hence, the quantity $h(G)-h(X)$ is a strong measure of closeness to Gaussianity. In fact, we have $D_{KL}(X \| G) = h(G)-h(X)$, where $D_{KL}$ is the so-called Kullback–Leibler divergence (or relative entropy).

In the celebrated article \cite{ABBN04} Artstein, Ball, Barthe and Naor showed that if $X_1, X_2, \ldots$ is a sequence of i.i.d. random variables with variance $1$ and $S_n = \frac{1}{\sqrt{n}}(X_1+\ldots+X_n)$, then the sequence $(h(S_n))_{n \geq 1}$ is nondecreasing. The convergence of this sequence to $h(G)$, where $G$ is a standard Gaussian random variable, was established much earlier by Barron in \cite{B86}, under minimal conditions that $h(S_{n_0})>-\infty$ for at least one $n_0 \ge 1$ (see also the work \cite{L59} of Linnik for some partial results). In view of these results the following natural problem arises.

\begin{prob}\label{prob:general}
For a given sequence $X_1, \ldots, X_n$ of i.i.d. random variables with finite second moment find the maximum of the function 
\[
S_+^{n-1} \ni (a_1,\ldots, a_n) \mapsto h\left(\sum_{i=1}^n a_i X_i\right).
\] 
What if $S_+^{n-1}$ is replaced with $S^{n-1}$? 
\end{prob}

\noindent Here by $S^{n-1}$ we denote the unit Euclidean sphere centered at the origin and  we take $S_+^{n-1}= S^{n-1} \cap [0,\infty)^n$. Note that if $X_i$ are i.i.d. then $\var(\sum_{i=1}^n a_i X_i) = \var(X_1) \sum_{i=1}^n a_i^2$ and hence in the above problem we are looking for the maximum of entropy of weighted sums of i.i.d. random variables under a fixed variance.

Before we state our main result we briefly introduce some notation. In this article $| \cdot|$ denotes the standard Euclidean norm and $\scal{\cdot}{\cdot}$ stands for the standard scalar product in $\R^n$.  By $\sim$ we denote equality in distribution of random variables. By Gamma $(\alpha, \beta)$ we mean a probability distribution admitting a density $\Gamma(\alpha)^{-1}\beta^{\alpha}x^{\alpha-1}e^{-\beta x}$ on $(0, \infty)$, and in case $\beta = 1$ we abbreviate it to Gamma($\alpha$). We also implicitly assume that in our abstract statements all integrals and expected values are well-defined and may have values $\pm \infty$. These statements are then used in very concrete settings where those quantities are easily seen to be well-defined and finite.

Our main result reads as follows.

\begin{thm}\label{prop:diagonal-chaos}
Let $X_1, \ldots, X_n$ be i.i.d.\ Gamma$(\alpha)$ random variables with $\alpha > 0$. Then for any integer $n\geq1/\alpha$ and any nonnegative real numbers $d_1, \ldots, d_n$ satisfying $\sum_{i=1}^n d_i^2 = 1$ one has
\[
	h\left( \sum_{i=1}^n d_i X_i \right) \leq h\left( \frac{1}{\sqrt{n}}\sum_{i=1}^n X_i \right), 
\]
with equality if and only if $d_1=\ldots= d_n=1/\sqrt{n}$.
\end{thm}

\noindent In the proof of this result we use the \emph{method of intersecting densities}, developed in \cite{ENT18-2} in the context of moment problems for log-concave random variables. This method is described in Section \ref{sec:intersecting} in a form suitable for our investigation. More detailed discussion of this method is given is Section \ref{sec:method}. 

Let us now discuss the state of the art of Problem \ref{prob:general}. One may ask whether or not the maximum in Problem \ref{prob:general} is achieved when $a_1=\ldots = a_n=1/\sqrt{n}$. Unfortunately, the answer  is negative even for symmetric random variables in the case $n=2$, as shown in \cite{BNT16}. In fact, solving Problem \ref{prob:general} is a difficult and complex issue even for the simplest random variables $X_i$. As an example, let us mention  $X_i$ being uniformly distributed in $[-1,1]$ (see Question \ref{que:uniform} in Section \ref{sec:op}), in which case it is believed that the maximum is attained for equal coefficients, but as far as we know it has not yet been proven. The only general result that we are aware of is Theorem 8 in \cite{ENT18}, where the problem was solved in the case of $X_i$ being i.i.d. Gaussian mixtures, that is, random variables of the form $X_i \sim R_i \cdot g_i$, where random variables $g_i \sim \mc{N}(0,1)$ and random variables $R_i>0$ are independent. In fact, the authors showed a stronger statement: if $X_i$ are i.i.d. Gaussian mixtures and $(a_1^2, \ldots, a_n^2) \prec (b_1^2, \ldots, b_n^2)$ in the Schur order, then $h(\sum_{i=1}^n a_i X_i) \geq h(\sum_{i=1}^n b_i X_i)$. Let us recall that the definition of the Schur order is that for vectors $(p_1, \ldots, p_n)$ and $(q_1, \ldots, q_n)$  we have $(p_1,\ldots, p_n) \prec (q_1,\ldots, q_n)$ iff $\sum_{i=1}^k p_i^\ast \leq \sum_{i=1}^k q_i^\ast$ for all $k=1,\ldots, n$ with equality for $k=n$, where $(p_i^\ast)$ and $(q_i^\ast)$ are nonincreasing rearrangements of the sequences $(p_i)$ and $(q_i)$. Note that for any $(a_1,\ldots, a_n) \in S^{n-1}$ we have $(1/n, \ldots, 1/n) \prec (a_1^2, \ldots, a_n^2) \prec (1,0,\ldots, 0)$, which shows that indeed in this case $a_1= \ldots = a_n= 1/\sqrt{n}$ gives the maximum in Problem \ref{prob:general}, whereas $a_1=1$, $a_2=\ldots=a_n=0$ gives the minimum. The latter is, in fact, true not only for Gaussian mixtures, but for arbitrary i.i.d. random variables $X_i$, which is an easy consequence of the famous entropy power inequality of Shannon and Stam (see \cite{S48, S59}) in the following linearized form: if the real numbers $a_i$ satisfy $\sum_{i=1}^n a_i^2=1$, then for a sequence of independent random variables $X_1, \ldots, X_n$ one has $h(\sum_{i=1}^n a_i X_i) \geq \sum_{i=1}^n a_i^2 h(X_i)$.

As an application of our main result, we study entropy of Gaussian quadratic forms. We introduce the following definition.

\begin{definition}
Let $G_n$ be a standard $\mc{N}(0,I_n)$ Gaussian random vector in $\R^n$ ($I_n$ stands for the $n \times n$ identity matrix). For a symmetric $n \times n$ real matrix $A$ we define $X_A = \scal{AG_n}{G_n}$. The random variable $X_A$ is called a \emph{Gaussian quadratic form} (in $n$ variables). If $A$ is additionally positive semi-definite, then $X_A$ is called a \emph{nonnegative Gaussian quadratic form}.
\end{definition}

\noindent Our main result easily gives the following corollary. 

\begin{cor}\label{thm:main}
Let $X_A$ be a nonnegative Gaussian quadratic form. Then
\[
	h\left( X_A \right) \leq h\left( \chi^2(n) \right) + \frac12 \ln  \var(X_A) - \frac12 \ln (2n)
\]
with equality if and only if $A= \lambda I_n$ for some $\lambda>0$. Here $\chi^2(n)= |G_n|^2$ is a random variable with a chi-square distribution with $n$ degrees of freedom. Equivalently, if $G$ is a Gaussian random variable with the same variance as $X_A$, then $h(G)-h(X_A) \geq \frac12\ln(4 \pi e n)-h(\chi^2(n)) = \frac{2}{3n} + o(1/n)$. 
\end{cor}

\begin{rem}
Corollary \ref{thm:main} shows that, in a sense of relative entropy, a Gaussian random variable cannot be approximated by a nonnegative Gaussian quadratic form too well, that is, if $\var(X_A)=\var(G)$, then $D(X \| G) \geq \frac{2}{3n}+o(1/n)$.   
\end{rem}

\noindent As we shall explain in Section \ref{sec:reduction}, rotation invariance of the standard Gaussian random vector $G_n$ in $\R^n$ allows us to reduce Corollary \ref{thm:main}  to the case of diagonal quadratic forms with nonnegative entries.   Since for a standard  $\mc{N}(0,1)$ Gaussian random variable $g$ its square $g^2$ has the same distribution as $2X$, where $X$ is a Gamma$(1/2)$ random variable, we get the following fact, of which Corollary \ref{thm:main} is a simple consequence. 

\begin{cor}\label{cor:1}
Let $g_1, \ldots, g_n$ be independent standard Gaussian random variables. Then for any $n\geq 1$ and any nonnegative real numbers $d_1, \ldots, d_n$ satisfying $\sum_{i=1}^n d_i^2 = 1$ one has
\[
	h\left( \sum_{i=1}^n d_i g_i^2 \right) \leq h\left( \frac{1}{\sqrt{n}}\sum_{i=1}^n g_i^2 \right), 
\]
with equality if and only if $d_1=\ldots= d_n=1/\sqrt{n}$.
\end{cor}

\noindent Similarly, since Gamma$(1)$ is the same distribution as Exp($1$), we also get the following corollary.

\begin{cor}\label{cor:2}
Let $X_1, \ldots, X_n$ be independent standard
exponential random variables, i.e. random variables with densities $e^{-x}$ on $[0, \infty)$. Then for any $n\geq 1$ and any nonnegative real numbers $d_1, \ldots, d_n$ satisfying $\sum_{i=1}^n d_i^2 = 1$ one has
\[
	h\left( \sum_{i=1}^n d_i X_i \right) \leq h\left( \frac{1}{\sqrt{n}}\sum_{i=1}^n X_i \right), 
\]
with equality if and only if $d_1=\ldots= d_n=1/\sqrt{n}$.
\end{cor}

%----------------------------------

%----------------------------------

This article is organized as follows. In Section \ref{sec:reduction} we show how Corollary \ref{cor:1} implies Corollary \ref{thm:main}. In Section \ref{sec:intersecting} we describe our key \emph{method of intersecting densities}. The proof of Theorem \ref{prop:diagonal-chaos} is given in Section \ref{sec:gaussy}. In Section \ref{sec:capacity}, as an application of our main result,  we derive a bound on the capacity of transmission channels subject to $n$ independent additive noises having gamma distribution. In Section \ref{sec:op} we present some open problems.  Section \ref{sec:diss} is devoted to further motivations and discussion of connections to the existing literature. Finally, a detailed description of possible strengths and weaknesses of our method is given in Section \ref{sec:method}.  

\section{Proof of Corollary \ref{thm:main}}\label{sec:reduction}

We begin with the following simple lemma.

\begin{lemma}\label{lem:rotation}
Let $X_A$ be a Gaussian quadratic form in $n$ variables and let $U$ be an orthogonal transformation in $\R^n$. Then $X_{U^\ast A U}$ has the same distribution as $X_A$. In particular, every Gaussian quadratic form has the same distribution as a certain Gaussian quadratic form with $A$ being diagonal. If additionally $X_A$ was assumed to be nonnegative, then the associated diagonal matrix has nonnegative entries.
\end{lemma}

\begin{proof}
Let $G_n \sim \mc{N}(0,I_n)$. Note that because of rotation invariance of $G_n$, the random vector $G_n'=U G_n$ has the same distribution as $G_n$. We have $X_{U^\ast A U} = \scal{U^\ast A U G_n}{G_n} = \scal{A U G_n}{U G_n} =\scal{A G_n'}{G_n'}$, which has the same distribution as $\scal{AG_n}{G_n}=X_A$. To prove the second part it suffices to observe that every symmetric matrix is diagonalizable by a certain orthogonal change of basis $U$. If the matrix $A$ is positive semi-definite, then the resulting diagonal matrix clearly has nonnegative entries.

\end{proof}

\begin{lemma}\label{lem:variance}
Let $X_A$ be a Gaussian quadratic form. Then $\E X_A = \tr(A)$ and $\var(X_A)=2\tr(A^2)$.
\end{lemma}

\begin{proof}
By Lemma \ref{lem:rotation} and by the invariance of $\tr(A)$ and $\tr(A^2)$ under matrix similarity, the statement is invariant under the transformation $A \to U^\ast A U$ for any orthogonal matrix $U$. We can therefore assume that $A$ is diagonal. In this case $X_A = \sum_{i=1}^n a_{ii} g_i^2$, where $a_{ii}$ are some real numbers and $g_i$ are i.i.d. $\mc{N}(0,1)$ random variables. Clearly, $\E X_A = \sum_{i=1}^n a_{ii} = \tr(A)$. Moreover,
\begin{align*}
	\E X_A^2 & = \E \Big( \sum_{i=1}^n a_{ii} g_i^2 \Big)^2 = \sum_{i=1}^n a_{ii}^2 \E g_i^4 + \sum_{i \ne j} a_{ii} a_{jj} \E g_i^2 g_j^2  = 3 \sum_{i=1}^n a_{ii}^2 + \sum_{i \ne j} a_{ii} a_{jj} \\ & = 2 \sum_{i=1}^n a_{ii}^2 + \Big(  \sum_{i=1}^n a_{ii}\Big)^2 = 2 \tr(A^2) + (\tr(A))^2.
\end{align*}
Hence, $\var(X_A) = \E X_A^2 - (\E X_A)^2 = 2\tr(A^2)$.

\end{proof}

Now we show how Corollary \ref{cor:1} implies Corollary \ref{thm:main}.

\begin{proof}[Proof of Corollary \ref{thm:main}]
Thanks to Lemma \ref{lem:rotation}, we can assume that $A$ is diagonal, that is, $X_A = \sum_{i=1}^n d_i g_i^2$ for some $d_i \geq 0$.  Since for any random variable $X$ and any non-zero real number $\lambda$ one has $h(\lambda X)=h(X)+\ln|\lambda|$ and $\var(\lambda X) = \lambda^2 \var(X)$, the statement is invariant under scaling of $X_A$. Thus, one can also assume that $\sum_{i=1}^n d_i^2 = 1$. In this case, due to Lemma \ref{lem:variance}, one has $\var(X_A)= 2$. Hence, Corollary \ref{cor:1} yields that $h\left( \chi^2(n) \right) + \frac12 \ln  \var(X_A) - \frac12 \ln (2n) = h(\chi^2(n))-\frac12 \ln n = h(n^{-1/2} \chi^2(n)) \geq h(X_A)$.  The equality cases follow easily from Lemma \ref{lem:rotation} and equality cases in Corollary \ref{cor:1}.   
\end{proof}

\section{The method of intersecting densities}\label{sec:intersecting}

We begin by recalling the following standard bound for the entropy.

\begin{lemma}\label{lem:entropy-standard-bound}
Suppose $p,q$ are probability densities of random variables $U$ and $V$, respectively. Take $\Phi=-\ln q$. Then 
\begin{itemize}
	\item[(a)] $-\int p \ln p \leq - \int p \ln q$, that is, $h(U) \leq \E\Phi(U)$,
	\item[(b)] if $\E \Phi(U) \leq \E \Phi(V)$, then $h(U) \leq h(V)$.  
\end{itemize}
  
\end{lemma}

\begin{proof}
(a) We can assume that the support of $p$ is contained in the support of $q$ (otherwise the right-hand side is $+\infty$ and there is nothing to prove). Since for $x \geq 0$ we have $\ln x \leq x-1$, one gets
\begin{align*}
	-\int p \ln p  + \int p \ln q & = \int_{\supp(p)} (p \ln q- p \ln p) = \int_{\supp(p)} p \ln(q/p) \\ & \leq \int_{\supp(p)} p(q/p-1) \leq 0. 
\end{align*}

(b) From part (a) we have $h(U) \leq \E \Phi(U) \leq \E \Phi(V) = -\int q \ln q = h(V)$.  
\end{proof}

In our proof of Theorem \ref{prop:diagonal-chaos}, in order to verify the assumption $\E \Phi(U) \leq \E \Phi(V)$ of Lemma \ref{lem:entropy-standard-bound}(b), we shall use a trick that we call  \emph{the method of intersecting densities}. The next lemma describes this crucial idea. Let us first introduce the following definition. 

\begin{definition}
Let $f : \R \to \R$ be measurable function. We say that $f$ \emph{changes sign} at point $x \in \R$ if one of the following conditions holds:
\begin{itemize}
	\item[(a)] there exist $y, z \in \R$ such that $y < x < z$ and $f$ is positive a.e. on $(x, z)$, nonpositive a.e. on $(y, x)$ and negative on some subset of $(y, x)$ of positive measure;
	\item[(b)] there exist $y, z \in \R$ such that $y < x < z$ and $f$ is negative a.e. on $(x, z)$, nonnegative a.e. on $(y, x)$ and positive on some subset of $(y, x)$ of positive measure.  
\end{itemize}
We call such $x$ the \emph{sign change point} of $f$. If $f$ has precisely $n$ sign change points, then we say that $f$ \emph{changes sign exactly $n$ times}.  
\end{definition}

Let us observe that if all sign change points of $f$ are $x_1, \dots, x_n$, then $f(x)(x - x_1)\dots(x - x_n)$ is either nonpositive a.e. on $\R$ or nonnegative a.e. on $\R$.

\begin{lemma}\label{lem:method-of-int-dens}
Suppose $\Phi(x)=-\tau \ln(x-L) + \alpha x^2 + \beta x + \gamma$, where $\alpha, \beta, \gamma, L$ are arbitrary real numbers and $\tau>0$. Suppose also that $U,V$ are real random variables with densities $f_U$ and $f_V$ supported in $[L,\infty)$, such that $\E U= \E V$, $\E U^2= \E V^2$, and the function $f_V-f_U$ changes sign exactly three times and is positive a.e. before the first sign change point. Then
$\E \Phi(U) < \E \Phi(V)$.    
\end{lemma}

\begin{proof}
Our goal is to prove the inequality $\int \Phi (f_V-f_U) > 0$. Because of our assumptions, we have $\int x^i f_U(x)\dd x = \int x^i f_V(x)\dd x$ for $i=0,1,2$. Our desired inequality is therefore equivalent to
\begin{equation}\label{eq:pos-int}
	\int_L^{\infty} (\Phi(x) - (a_2 x^2 + a_1 x + a_0))(f_V(x)-f_U(x)) \dd x > 0,
\end{equation}
where $a_0, a_1, a_2$ are arbitrary real numbers. A crucial step now is to explore the freedom of the choice of these three numbers. We know that $f_V-f_U$ changes sign exactly three times at some points $x_0<x_1<x_2$.  We choose $a_0,a_1,a_2$ so that $\Phi(x_i)=a_2 x_i^2 + a_1 x_i+a_0$ for $i=0,1,2$. This can be done because the matrix $(x_i^j)_{i,j=0}^2$, associated to the system of linear equations that $a_0,a_1,a_2$ have to satisfy, is a $3 \times 3$  Vandermonde matrix. 

Let $\Psi(x)=\Phi(x) - (a_2 x^2 + a_1 x + a_0)$. We now show that the integrand $\Psi(f_V-f_U)$ in \eqref{eq:pos-int} is nonnegative, which will clearly finish the proof (the obtained inequality will be strict because it will also easily follow that this integrand is not an a.e. zero function).  We already know that $\Psi(x_0)=\Psi(x_1)=\Psi(x_2)=0$ and that $f_V-f_U$ changes sign at $x_0, x_1, x_2$ and is positive a.e. before the first sign change point $x_0$. Since close to $x=L$ the function $\Psi(x)$ is positive (note that $\lim_{x \to L^+} \Psi(x)=+\infty$), it is enough to show that $\Psi$ also changes its sign at $x_0, x_1, x_2$ and that these are the only sign change points of this function. 

To show this we observe that the function $\Psi$ has the form $\Psi(x)= -\tau \ln(x-L)+ax^2+bx+c$ for some real numbers $a,b,c$. This function is clearly smooth on $(L,\infty)$. It is enough to show that $\Psi$ has only three zeros and none of them is a zero of $\Psi'$ (then we easily conclude that the zeros correspond to sign changes). Suppose that $\Psi$ has more than three zeros, counting multiplicities ($x$ is a zero of multiplicity $k$ if $\Psi^{(j)}(x)=0$ for $j=0,1,\ldots, k-1$, where $\Psi^{(j)}$ is the $j$th derivative of $\Psi$, with the convention that $\Psi^{(0)}=\Psi$). Since $\Psi$ itself has at least three distinct zeros, by Rolle's theorem we deduce that $\Psi'$ has at least three distinct zeros. But for $x>L$ we have
$
	\Psi'(x) = \frac{-\tau}{x-L} + 2ax+b.
$
Thus, the equation $\Psi'(x)=0$ is equivalent to the quadratic equation $(2 a x + b)(x-L)=\tau$, which cannot have more than two solutions (unless $\Psi'$ vanishes identically, which clearly  does not hold in our case as $\tau>0$). We arrived at a contradiction.
\end{proof}

\noindent In Lemma \ref{lem:method-of-int-dens} we assumed that $f_U-f_V$ changes sign exactly three times and that $\E U= \E V$ and $\E U^2 = \E V^2$. Our next lemma shows that the conditions $\E U= \E V$ and $\E U^2 = \E V^2$ are enough to guarantee that $f_U-f_V$ changes sign at least three times. 

\begin{lemma}\label{lem:interlacing}
Let $k,n \geq 1$ be integers and let $g: \R \to \R$ be measurable. Suppose that $g$ changes sign at exactly $k$ points. Assume moreover that $\int_{\R} x^j g(x) \dd x = 0$ for all $j=0,1,\ldots, n-1$. Then $k \geq n$.   
\end{lemma}

\begin{proof}
We prove the lemma by contradiction. Assume that $k \leq n-1$. Let $x_1< x_2 < \ldots < x_k$ be the sign change points of $g$. From our assumption, for every polynomial $P$ of degree at most $n-1$ one has $\int Pg=0$. Let us take $P(x)=(x-x_1)\ldots(x-x_k)$ and  $h=Pg$. We have $\int h  = 0$. On the other hand, $h$ does not change sign since $P$ changes sign exactly at the same points as $g$. Since $h$ is not identically zero, we get $\int h \ne 0$,  contradiction. 
\end{proof}

\begin{cor}\label{cor:three-times}
Suppose $U,V$ are real random variables with densities $f_U$ and $f_V$, such that $\E U= \E V$ and $\E U^2= \E V^2$. Then the function $f_V-f_U$ changes sign at least three times.
\end{cor}

\begin{proof}
It is enough to apply Lemma \ref{lem:interlacing} with $g=f_U-f_V$ and $n=3$.
\end{proof}

\section{Proof of Theorem \ref{prop:diagonal-chaos}}\label{sec:gaussy}

\begin{lemma}\label{lem:continuity}
Suppose that $Y$ is a random vector having values in $(-l, \infty)^n$, where $l \geq 0$. Let $\Phi: (-l \sqrt{n}, \infty) \to \R$ be continuous and assume that there exists a measurable function $M:(-l,\infty)^n \to \R$, such that one has $\sup_{x \in S_+^{n-1}}|\Phi(\scal{x}{y})| \leq M(y)$ and $\E M(Y)<\infty$. Then, the function $F:S_+^{n-1} \to \R$ defined by $F(x) = \E \Phi(\scal{x}{Y})$ is continuous.  
\end{lemma}

\begin{proof}
We have
$
F(x) =  \int_{(-l,\infty)^n} \Phi(\scal{x}{y}) \dd \mu(y),
$ 
where $\mu$ is the law of $Y$. Suppose $x^{(n)} \to x$. Then, $\Phi(\scal{x^{(n)}}{y}) \to \Phi(\scal{x}{y})$ for any fixed $y \in (-l,\infty)^n$, by the continuity of $\Phi$ (note that $\scal{x}{y} > -l \sqrt{n}$). The assertion follows from the Lebesgue dominated convergence theorem.    
\end{proof}

\begin{lemma}\label{lem:wyrownywanie-wystarczy}
Suppose that $Y=(Y_1, \ldots, Y_n)$, where $Y_1, \ldots, Y_n$ are i.i.d. random variables having values in the interval $(-l, \infty)$, where $l \geq 0$. Let $\Phi: (-l \sqrt{n}, \infty) \to \R$ be such that the function $F(x)=\E \Phi(\scal{x}{Y})$ is continuous. Suppose that for every $0 \leq d_1 < d_2$ satisfying $d_1^2+d_2^2 \leq 1$  we have
\begin{equation}\label{eq:zblizanie-dla-phi-2}
\E \Phi(s+d_1 Y_1 + d_2Y_2) \leq \E \Phi\left(s+\sqrt{\frac12(d_1^2+d_2^2)}(Y_1 + Y_2) \right), 
\end{equation} 
for every  $s \geq - l \sqrt{(n-2)(1-d_1^2-d_2^2)}$. Then,
\begin{equation}\label{eq:phi-maximizer-2}
\E \Phi\left(\sum_{i=1}^n d_i Y_i \right) \leq \E \Phi\left(\frac{1}{\sqrt{n}}\sum_{i=1}^n Y_i \right)
\end{equation}
for all $(d_1,\ldots, d_n) \in S_+^{n-1}$. Moreover, for $(d_1, \ldots, d_n) \ne (n^{-1/2},\ldots, n^{-1/2})$, if \eqref{eq:zblizanie-dla-phi-2} is always strict, then the inequality \eqref{eq:phi-maximizer-2} is strict.
\end{lemma}

\begin{proof}
We first show that if $d_1, d_2 \geq 0$ satisfy $d_1^2+d_2^2=\delta^2$ and $s \geq - l \sqrt{(n-2)(1-\delta^2)}$, then $\E \Phi \left( s + d_1 Y_1 + d_2 Y_2\right)$ is well-defined, that is, $s+d_1 Y_1 + d_2 Y_2 > -l \sqrt{n}$ is a.s. satisfied.  Indeed,
\begin{align*}
	s+ d_1 Y_1 + d_2 Y_2 & > -l \sqrt{(n-2)(1 - \delta^2)} - l(d_1+d_2) 
	 \\ & \geq -l\sqrt{(n-2)(1 - \delta^2)} -l\sqrt{2(d_1^2+d_2^2)}  \\
	& = -l\left( \sqrt{n-2} \cdot \sqrt{1-\delta^2} + \sqrt{2} \delta \right)  
	  \\ & \geq - l \sqrt{n-2+2} \cdot \sqrt{1-\delta^2 + \delta^2} = - l \sqrt{n}, 
\end{align*}
where the last inequality results from the Cauchy-Schwarz inequality. 

We now consider a random variable $S=\sum_{i=3}^n d_i Y_i$, where $\sum_{i=3}^n d_i^2=1-\delta^2$. Observe that a.s.
\[
	S > - l \sum_{i=3}^n d_i \geq -l \sqrt{(n-2) \sum_{i=3}^n d_i^2} = -l \sqrt{(n-2) (1-\delta^2)},  
\] 
again by Cauchy-Schwarz inequality. Substituting $S$ for $s$ in \eqref{eq:zblizanie-dla-phi-2} and taking expectation with respect to $Y_3, \ldots, Y_n$ leads to the inequality
\[
	\E \Phi\left( \sum_{i=1}^n d_i Y_i \right) \leq \E \Phi\left( \sqrt{\frac12(d_1^2+d_2^2)}(Y_1 + Y_2) + \sum_{i=3}^n d_i Y_i \right),
\]
where $(d_1, \ldots, d_n) \in S_+^{n-1}$.
Moreover, if  \eqref{eq:zblizanie-dla-phi-2} is strict for every $s$, then the above inequality is also strict. 

Since the function $F$ is continuous and $S_+^{n-1}$ is compact, $F$ achieves its maximum in some point $(d_1, \ldots, d_n) \in S_+^{n-1}$. We first show the \emph{moreover} part of the assertion. If  \eqref{eq:zblizanie-dla-phi-2} is strict,  then for $d_1 \ne d_2$ we have 
\begin{equation}\label{eq:F}
F(d_1, d_2, d_3, \ldots, d_n) <  F\left(\sqrt{(d_1^2+d_2^2)/2}, \sqrt{(d_1^2+d_2^2)/2}, d_3,\ \ldots, d_n\right).
\end{equation}
Now, if $(d_1, \ldots, d_n)$ has two different coordinates, then using permutation invariance of $F$ we can assume that these coordinates are $d_1 \ne d_2$ and \eqref{eq:F} immediately gives a contradiction.  

We now show the first part. Let $A \subseteq S_+^{n-1}$ be the set where the maximum $m$ of $F$ is achieved. Since $F$ is continuous, this set is compact. The function $g:A \to \R$ defined by $g(x_1, \ldots, x_n) = x_1+\ldots+x_n$ achieves its maximum on $A$. We claim that the point where the maximum is achieved must be the point $(n^{-1/2},\ldots, n^{-1/2})$. If it is some other point $(d_1,\ldots, d_n)$, then, without loss of generality, $d_1 \ne d_2$ and we observe that
\begin{align*}
 m & = F(d_1, d_2, d_3, \ldots, d_n) \\ & \leq   F\left(\sqrt{(d_1^2+d_2^2)/2}, \sqrt{(d_1^2+d_2^2)/2}, d_3,\ \ldots, d_n\right) \leq m.
\end{align*}
Thus, $\left(\sqrt{(d_1^2+d_2^2)/2}, \sqrt{(d_1^2+d_2^2)/2}, d_3,\ \ldots, d_n\right) \in A$ and clearly the sum of coordinates of this vector is strictly bigger than the sum of coordinates of $(d_1, \ldots, d_n)$, which gives a contradiction with the fact that the latter vector maximized $g$ on $A$.
\end{proof}

\begin{lemma}\label{lem:technical-density-bessel}
Assume that $X_1, X_2$ are i.i.d.\ Gamma$(\alpha)$ real random variables, let $d_1,d_2>0,c_1 = \Gamma(\alpha)^{-2} (d_1 d_2)^{-\alpha}, c_2=\sqrt{(d_1^2+d_2^2)/2}$ and define $I_{d_1, d_2}(x) = \int_0^1 [t(1-t)]^{\alpha-1} e^{-x\left(\frac{t}{d_1}+\frac{1-t}{d_2}\right)}\dd t$. Then,
\begin{itemize}
\item[(a)] the density of $d_1 Y_1 + d_2 Y_2$, where $Y_i=X_i-\alpha$, $i=1,2$, is equal to
\[
	f_{d_1, d_2}(x) =
	c_1(x+\alpha(d_1+d_2))^{2\alpha-1} I_{d_1, d_2}(x+\alpha(d_1+d_2)) \1_{[-\alpha(d_1+d_2),\infty)}(x), 
\]
\item[(b)] the density of $\sqrt{\frac12\left(d_1^2+d_2^2\right)}( Y_1 + Y_2)$, where $Y_i=X_i-\alpha$, $i=1,2$, is equal to
\[
	g_{d_1, d_2}(x) = \frac{1}{\Gamma(2\alpha)c_2^{2\alpha}}\Big(x+\alpha\sqrt{2(d_1^2+d_2^2)}\Big)^{2\alpha-1}e^{-\left(\frac{x}{c_2}+2\alpha\right)} \1_{[-\alpha\sqrt{2(d_1^2+d_2^2)},\infty)}(x).
\]
\end{itemize}
\end{lemma}

\begin{proof}
Let $\rho_{d_1, d_2}$ be the density of $d_1 X_1+ d_2 X_2$. Clearly $\rho_{d_1, d_2}$  is supported in $[0,\infty)$. The density of $X_1$ is equal to $f(u)=\frac{1}{\Gamma(\alpha)}u^{\alpha-1}e^{-u}\1_{[0,\infty)}(u)$. As a consequence, for $x \geq 0$ we have 
\begin{align*}
	\rho_{d_1, d_2}(x) 
	  & = \frac{1}{d_1 d_2}\int f\left( \frac{u}{d_1}\right) f\left( \frac{x-u}{d_2}\right) \dd u \\
	  & = \frac{1}{\Gamma(\alpha)^2 (d_1 d_2)^\alpha} \int_0^x u^{\alpha-1} (x-u)^{\alpha-1} e^{-\left(\frac{u}{d_1}+\frac{x-u}{d_2}\right)} \dd u \\
	  & \overset{u=xt}{=} c_1 \int_0^1 (xt)^{\alpha-1} (x(1-t))^{\alpha-1} e^{-x\left(\frac{t}{d_1}+\frac{1-t}{d_2}\right)} x\dd t \\
	  & = c_1 x^{2\alpha-1} \int_0^1 [t(1-t)]^{\alpha-1} e^{-x\left(\frac{t}{d_1}+\frac{1-t}{d_2}\right)} \dd t \\
	  & = c_1 x^{2\alpha-1}I_{d_1, d_2}(x).
\end{align*}
By shifting we obtain $f_{d_1, d_2}$. To get the formula for $g_{d_1, d_2}$ note that $X_1 + X_2 \sim \text{Gamma}(2\alpha)$ and apply a suitable affine map.
\end{proof}

\begin{lemma}\label{lem:3-times}
Assume that $X_1, X_2$ are i.i.d.\ Gamma$(\alpha)$ real random variables with $\alpha \in (0, 1/2]$  and let $d_1,d_2 \geq 0$ be such that $d_1 \ne d_2$. Let $f_{U}$ be the density of $U=d_1 Y_1 + d_2 Y_2$, where $Y_i=X_i-\alpha$, $i=1,2$ and let $f_{V}$ be the density of $V=\sqrt{\frac12(d_1^2+d_2^2)}( Y_1 + Y_2)$. Then there exist points $x_0<x_1<x_2$ such that $f_{V}-f_{U}$  is strictly positive on $(-\alpha\sqrt{2(d_1^2+d_2^2)}, x_0) \cup(x_1,x_2)$ and strictly negative on $(x_0,x_1) \cup (x_2, \infty)$. Moreover, $x_0=-\alpha(d_1+d_2)$.
\end{lemma}

\begin{proof}
Suppose first that $d_1, d_2 \ne 0$. Since $d_1 \ne d_2$, we have $\alpha\sqrt{2(d_1^2+d_2^2)}>\alpha(d_1+d_2)$. Therefore, the support of $d_1 Y_1 + d_2 Y_2$ is strictly contained in the support of $\sqrt{\frac12(d_1^2+d_2^2)}( Y_1 + Y_2)$. Thus, $f_{V}(x)>0 = f_{U}(x)$ on the interval $(-\alpha\sqrt{2(d_1^2+d_2^2)}, -\alpha(d_1+d_2))$, so on this interval one has $f_{V} - f_{U}>0$. For $\alpha \in (0, 1/2)$ we now observe that $\lim_{x\rightarrow x_0^+}f_U(x) = \infty$ and for $\alpha=1/2$ we calculate 
\begin{align*}
(f_{V} - f_{U})(-\alpha(d_1+d_2)) &=  \sqrt{\frac{2}{d_1^2+d_2^2}} \exp \left(\frac{d_1+d_2}{\sqrt{2(d_1^2+d_2^2)}} -1 \right) - \frac{1}{\sqrt{d_1 d_2}}  \\
& < \sqrt{\frac{2}{d_1^2+d_2^2}} - \frac{1}{\sqrt{d_1d_2}} < 0.
\end{align*}
Since both $f_{V}$ and $f_{U}$ are continuous on $(-\alpha(d_1+d_2),\infty)$, we see that $f_{V} - f_{U}$ is negative at least on some right neighborhood of $-\alpha(d_1+d_2)$. As a consequence, $x_0=-\alpha(d_1+d_2)$ is the first sign change point. 

We now argue that in $(x_0,\infty)$ there are at most two sign change points. 
Indeed, the sign of $f_{V} - f_{U}$ is the same as the sign of $\ln f_{V} - \ln f_{U}$. The functions $\ln f_{V}$ and $\ln f_{U}$, defined on $(x_0,\infty)$, are of the form 
\[
    \ln f_{V}(x) = (2\alpha-1)\ln\left(x+\alpha\sqrt{2(d_1^2+d_2^2)}\right) + a(x),
\]
\[
    \ln f_{U}(x) = C + (2\alpha-1)\ln\bigg(x+\alpha(d_1+d_2)\bigg) + \ln I_{d_1, d_2}(x+\alpha(d_1+d_2)),
\]
with $a(x)$ being an affine function and $C$ being a constant. Let us prove that $\ln f_{V} - \ln f_{U}$ is strictly concave. The difference 
\[
    (2\alpha-1)\ln\Big(x+\alpha\sqrt{2(d_1^2+d_2^2)}\Big) - (2\alpha-1)\ln\Big(x+\alpha(d_1+d_2)\Big)
\]
is concave as its second derivative is equal to 
\[
    (2\alpha-1)\left(-\frac{1}{\Big(x+\alpha\sqrt{2(d_1^2+d_2^2)}\Big)^2} + \frac{1}{\Big(x+\alpha(d_1+d_2)\Big)^2}\right),
\]
which is nonpositive because $\alpha\sqrt{2\left(d_1^2+d_2^2\right)} > \alpha\left(d_1+d_2\right)$ and $2\alpha-1\leq0$.
The function $u(x)=\ln I_{d_1, d_2}(x)= \ln\left(\int_0^1 [t(1-t)]^{\alpha-1} e^{-x\left(\frac{t}{d_1}+\frac{1-t}{d_2}\right)}\dd t\right)$ is strictly convex (the inequality $I_{d_1, d_2}(\lambda x + (1-\lambda)y ) < I_{d_1, d_2}(x)^\lambda I_{d_1, d_2}(y)^{1-\lambda}$ for $x \ne y$ follows from the H\"older inequality). Hence, $\ln f_{V} - \ln f_U$ is a sum of an affine function, a concave function and a composition of a strictly concave function with an affine function. Thus, it is a strictly concave function. As a consequence, it cannot have more than two zeros in $(x_0,\infty)$, which means that we can find appropriate points $x_1$ and $x_2$. Altogether, the number of sign change points of $f_{V} - f_{U}$ does not exceed three.     

If, say, $d_1=0$ and $d_2>0$, then one easily checks that
\[
	f_U(x) = C_U (x+\alpha d_2)^{\alpha-1}e^{-(x/d_2+\alpha)}  \1_{[-\alpha d_2,\infty)}(x)
\]
and
\[ 
    f_{V}(x) = C_V \left(x+\alpha d_2 \sqrt2\right)^{2\alpha-1}e^{-\left(x\sqrt{2}/d_2+2\alpha\right)} \1_{\left[-\alpha d_2 \sqrt{2},\infty\right)}(x),
\]
where $C_U, C_V$ are some constants. Again, $f_{V}-f_{U}$ is positive on the interval $(-\alpha d_2 \sqrt{2}, -\alpha d_2)$ and negative on some right neighborhood of $x_0=-\alpha d_2$. That is because $\lim_{x \to x_0^+} f_{U}(x)=+\infty$ for $\alpha\in(0,1/2]$. On $(-\alpha d_2, \infty)$ the function $\ln f_{V}- \ln f_{U}$ can be written as
\[
    a(x) + (2\alpha-1)\ln(x+\alpha d_2 \sqrt2) - (2\alpha-1)\ln(x+\alpha d_2) + \alpha\ln(x+\alpha d_2)
\]
for some affine function $a(x)$. Similarly to the previous case, two middle components form a concave function and hence $\ln f_{V}- \ln f_{U}$ is again strictly concave. Thus, it can have at most two sign changes in $(-\alpha d_2,\infty)$, and therefore at most three sign changes in the whole real line.  

Since $\E U = \E V$ and $\E U^2 = \E V^2$, the function $g_{d_1, d_2}-f_{d_1, d_2}$, by Corollary \ref{cor:three-times}, changes sign at at least three points. Thus, it has  precisely three sign change points.
\end{proof}

We are now ready to prove Theorem \ref{prop:diagonal-chaos}. A scheme of the proof is illustrated in Figure \ref{fig:1}.

\begin{figure}[h]
\begin{tikzpicture}[>=latex,rect/.style={draw=black, 
                   rectangle, 
                   fill=white,
                   fill opacity = 0.2,
                   text opacity=1,
                   minimum width=20pt, 
                   minimum height = 30pt, 
                   align=center}]
  \node[rect] (a1) {Lemma \ref{lem:interlacing} };
  \node[rect,above=20pt of a1] (a2) {Corollary \ref{cor:three-times}};
  \node[rect,above=20pt of a2] (a3) {Lemma \ref{lem:3-times} };
  \node[rect,right=20pt of a2] (a4) {Lemma \ref{lem:technical-density-bessel} }; 
  \node[rect,above right=50pt of a3] (a5) {Theorem \ref{prop:diagonal-chaos} };
  \node[rect,right=30pt of a3] (a6) {Lemma \ref{lem:method-of-int-dens} };
  \node[rect,right=30pt of a6] (a9) {Lemma \ref{lem:continuity}  \\ Lemma \ref{lem:wyrownywanie-wystarczy} }; 
  \node[rect,right=30pt of a9] (a7) {Lemma \ref{lem:entropy-standard-bound}(b)};
  \node[rect,below=30pt of a7] (a8) {Lemma \ref{lem:entropy-standard-bound}(a)};
  \draw[->] (a1)--(a2)node[midway]{};
  \draw[->] (a2)--(a3)node[midway]{};
  \draw[->] (a3)--(a5)node[midway]{};
  \draw[->] (a4.north)--(a3)node[midway]{};  
  \draw[->] (a6.north)--(a5)node[midway]{};
  \draw[->] (a8)--(a7)node[midway]{};
  \draw[->] (a7.north)--(a5)node[midway]{};
  \draw[->] (a9.north)--(a5)node[midway]{};
\end{tikzpicture}
\caption{Scheme of the proof of Theorem 1.}\label{fig:1}
\end{figure}
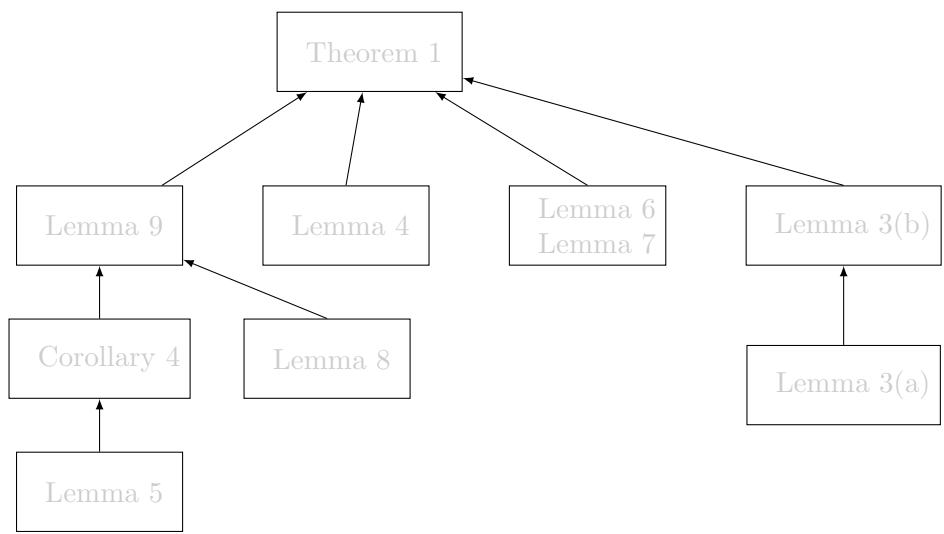

\begin{proof}[Proof of Theorem \ref{prop:diagonal-chaos}]
Since for any random variable $X$ and any real number $s$ we have $h(s+X)=h(X)$, we can replace the random variables $X_i$ with $Y_i=X_i-\E X_i = X_i - \alpha$. 
Proof can be reduced to $\alpha\in(0,1/2]$ case as sum of $k$ independent copies of Gamma$(\alpha)$ random variable follows Gamma$(k\alpha)$ distribution. 
Indeed, if we show that equal weights in a weighted sum of $n' = nk$ independent Gamma$(\alpha)$ random variables maximise the entropy, then the same result for the weights in a weighed sum of $n$ independent Gamma$(k\alpha)$ random variables will follow. Note that the assumption $n' \geq 1/\alpha$, needed to use the case of $n'$ summands and Gamma$(\alpha)$ random variables, is equivalent to $n \geq \frac{1}{k\alpha}$, which is the assumption for $n$ summands and Gamma$(k\alpha)$ random variables.  
By Lemma \ref{lem:entropy-standard-bound}(b), Lemma \ref{lem:continuity} and Lemma \ref{lem:wyrownywanie-wystarczy}, it is enough to show that 
\begin{itemize}
\item[(a)]
for $0 \leq d_1 <  d_2$ satisfying $d_1^2+d_2^2 \leq 1$ one has 
\[
	\E \Phi_n (s+d_1 Y_1 + d_2 Y_2) < \E \Phi_n \left(s+ \sqrt{(d_1^2+d_2^2)/2}(Y_1 + Y_2) \right)
\] 
for every $s \geq -\alpha \sqrt{(n-2)(1-d_1^2-d_2^2)}$ (note that in our case $l=\alpha$), where on $(-\alpha\sqrt{n},\infty)$ we have $\Phi_n = - \ln p_n$ and $p_n$ is the density of $\frac{1}{\sqrt{n}} \sum_{i=1}^n Y_i = \frac{1}{\sqrt{n}}( \text{Gamma}(n\alpha)-n\alpha)$, 
\item[(b)] there exists a measurable function $M:(-\alpha,\infty)^n \to \R$ such that $\sup_{x \in S_+^{n-1}} |\Phi_n(\scal{x}{y})| \leq M(y)$ and $\E M(Y_1,\ldots, Y_n) < \infty$.
\end{itemize}

We first verify (a). The density $q_n$ of Gamma$(n\alpha)$ is equal to 
\[
q_n(x) = \frac{1}{\Gamma(n\alpha)} x^{n\alpha-1} e^{-x} \1_{[0,\infty)}(x),
\]
 hence the density of $\frac{1}{\sqrt{n}} (\text{Gamma}(n\alpha)-n\alpha)$ is equal to
\begin{align*}
	p_n(x) 
	  & = \sqrt{n}q_n(\sqrt{n}x+n\alpha) \\ 
	  & =  \frac{\sqrt{n}}{\Gamma(n\alpha)} (\sqrt{n}x+n\alpha)^{n\alpha-1} e^{-\sqrt{n}x - n\alpha} \1_{[0,\infty)}(\sqrt{n}x + n\alpha) \\
	  & =  \frac{n^{n\alpha/2}}{\Gamma(n\alpha)} (x + \alpha\sqrt{n})^{n\alpha  -1} e^{-\sqrt{n}x - n\alpha} \1_{[-\alpha\sqrt{n},\infty)}(x).
\end{align*}
Therefore, for $x \geq -\alpha \sqrt{n}$ we have
$
	\Phi_n(x) =  \left( 1 - n\alpha \right) \ln (x + \alpha\sqrt{n}) + \sqrt{n}x  + c_n,
$ 
where $c_n$ depends only on $n$. By Lemma \ref{lem:method-of-int-dens}, it is enough to show that the difference $f_V-f_U$ of the densities $f_U$ and $f_V$ of $U=s+d_1 Y_1 + d_2 Y_2$ and $V=s+\sqrt{(d_1^2+d_2^2)/2}(Y_1+Y_2)$, respectively, changes sign exactly three times and is positive a.e. before the first sign change point. This is guaranteed by Lemma \ref{lem:3-times}.    

Let us now show (b). The affine part of $\Phi_n$ can easily be handled, since for $x \in S_+^{n-1}$ we have $\abs{\scal{x}{y}} \leq |y|$ and $\E \abs{(Y_1,\ldots, Y_n)} \leq \sum_{i=1}^n \E|Y_i| = n \E |Y_1|<\infty$. We have to bound $\abs{\ln(\scal{x}{y}+\alpha\sqrt{n})}$ for $y \in (-\alpha,\infty)^n$. Note that if $t \geq 1$, then $\abs{\ln t} = \ln t \leq t-1 \leq t$ and if $0<t<1$, then $\abs{\ln t} = -\ln t=b \ln (t^{-1/b}) \leq b(t^{-1/b}-1) \leq b t^{-1/b}$ for any $b>0$. Thus, $\abs{\ln t} \leq t+ b t^{-1/b}$ for every $t,b>0$ and we get
\[
	\abs{\ln(\scal{x}{y}+\alpha\sqrt{n})} \leq \abs{\scal{x}{y}} + \alpha\sqrt{n} + b (\scal{x}{y}+\alpha\sqrt{n})^{-1/b}.
\]
The term $\abs{\scal{x}{y}}$ can be bounded as above and $\alpha\sqrt{n}$ is a finite constant. We are therefore left with the term $(\scal{x}{y}+\alpha\sqrt{n})^{-1/b}$. 

Suppose that  $ \min_i x_i \le 1/(2 \sqrt{n})$, say, $x_1 \leq 1/(2 \sqrt{n})$. Then 
\[
\sum_{i=1}^n x_i \leq x_1 + \sqrt{n-1}(1-x_1^2) \leq \frac{1}{2 \sqrt{n}} + \sqrt{n-1}. 
\] 
Define $c_n = \sqrt{n}-\sqrt{n-1}- \frac{1}{2 \sqrt{n}}$. We have $c_n>\sqrt{n}-\sqrt{n-1}- \frac{1}{\sqrt{n-1}+ \sqrt{n}}=0$. Hence, $\scal{x}{y}+\alpha\sqrt{n} \geq \alpha\sqrt{n}-\alpha\sum_{i} x_i  \geq \alpha c_n$, so $(\scal{x}{y}+\alpha\sqrt{n})^{-1/b} \leq c_n^{-1/b}$, uniformly with respect to $y$. 

If for all $i$ we have $x_i > 1/(2\sqrt{n})$, then 
\begin{align*}
  (\scal{x}{y}+\alpha\sqrt{n})^{-1/b} & \leq \left(\sum_{i=1}^n x_i(y_i+\alpha) \right)^{-1/b} \leq (2\sqrt{n})^{1/b} \left(\sum_{i=1}^n (y_i+\alpha) \right)^{-1/b} \\ & \leq (2\sqrt{n})^{1/b} (y_1+\alpha)^{-1/b}.
\end{align*}
Now we set $b=2/\alpha$. Since 
\begin{align*}
    \E(Y_1+\alpha)^{-1/b} = \E X_1^{-\alpha/2} = \frac{1}{\Gamma(\alpha)}\int_0^\infty x^{-\alpha/2}x^{\alpha-1}e^{-x}\dd x = \frac{\Gamma(\alpha/2)}{\Gamma(\alpha)}<\infty,
\end{align*}the proof is completed.
\end{proof}

\section{Channel capacities}\label{sec:capacity}

Consider a memoryless transmission channel with power budged $P$ subject to additive noise $N$. If $X$ is an input of the channel then the output produced by the channel at the receiver is $Y=X+N$, where $X$ and $N$ are independent. The \emph{capacity} of the channel is given by the famous channel coding theorem of Shannon (see \cite{S48}):
\[
	C_P(N) = \sup_{X: \ \var(X) \leq P} (h(X+N) - h(N)).
\]  
Let $P_N = \var(N)$ be the noise power. Shannon (see \cite{S48}, Theorem 18) gave the following bounds 
\begin{equation}\label{eq:shannon}
	\frac12 \ln\left( 1+\frac{P}{\mc{N}(N)} \right) \leq C_P(N) \leq \frac12 \ln\left( \frac{P+P_N}{\mc{N}(N)} \right),
\end{equation}
where $\mc{N}(N)=\frac{1}{2\pi e} \exp(2h(N))$ is the so-called \emph{entropy power}. It is straightforward to check that the right inequality follows from the fact that Gaussian densities maximize entropy under fixed variance, which can equivalently be written as $h(Y)\leq \frac12 \ln(2\pi e \var(Y))$ or $\mc{N}(Y) \leq \var(Y)$. This gives  
\begin{align*}
h(X+N)-h(N) & = \frac12 \ln\left( \frac{\mc{N}(X+N)}{\mc{N}(N)} \right) \\ & \leq \frac12 \ln\left( \frac{\var(X+N)}{\mc{N}(N)} \right)  =\frac12 \ln\left( \frac{P+P_N}{\mc{N}(N)} \right).
\end{align*}
The left inequality is a consequence of the entropy power inequality, namely
\begin{align*}
	\sup_{X: \ \var(X) \leq P} (h(X+N) - h(N)) & = \frac12 \sup_{X: \ \var(X) \leq P} (\ln \mc{N}(X+N) - \ln \mc{N}(N)) \\ & \geq \frac12 \sup_{X: \ \var(X) \leq P} (\ln (\mc{N}(X)+\mc{N}(N)) - \ln \mc{N}(N)) \\
	& =  \frac12  (\ln (P+\mc{N}(N)) - \ln \mc{N}(N)) \\ & = \frac12 \ln\left( 1+\frac{P}{\mc{N}(N)} \right).
\end{align*}

Suppose that the noise $N$ is of the form $N=\sum_{i=1}^n a_i X_i$, where $X_i \sim \textrm{Gamma}(\alpha)$. In other words, our noise has $n$ independent sources, each having gamma distribution with certain \emph{scale} $a_i$ and \emph{shape} $\alpha$. Theorem \ref{prop:diagonal-chaos} allows us to estimate capacity of this channel. Namely, we have the following corollary.

\begin{cor}\label{cor:capacity}
Let $n \geq 1/\alpha$. The capacity of the additive channel with power budget $P$ and noise of the form $N=\sum_{i=1}^n a_i X_i$, where $X_i \sim \textrm{Gamma}(\alpha)$,  and with noise power $P_N$ satisfies  
\[
	\frac12 \ln \left(1+ \frac{P}{P_N} \cdot \tau(n \alpha) \right) \leq C_P(N) \leq \frac12 \ln \left(\frac{P+P_N}{P_N }\tau(\alpha)  \right)
\]
with 
\[
\tau(x)=\mc{N}(x^{-\frac12}\textrm{Gamma}(x))^{-1} = 2\pi e \exp(-2x-2 \ln \Gamma(x)-2(1-x) \psi(x)+ \ln x), 
\]
where $\psi$ denotes the digamma function. The upper bound  holds for every $n \geq 1$ and $\alpha>0$.
\end{cor} 

\begin{proof}
Theorem \ref{prop:diagonal-chaos} yields
\[
\mc{N}\left(|a| X_1 \right) \leq \mc{N}(N) \leq \mc{N}\left(\frac{|a|}{\sqrt{n}} \sum_{i=1}^n X_i \right).
\]
Note that here the lower bound is a simple consequence of the entropy power inequality (this was already mentioned in Section \ref{sec:intro}). Since $\E X_i=\alpha$, we have $P_N = \E|\sum_{i=1}^n a_i(X_i-\alpha)|^2=|a|^2 \E|X_1-\alpha|^2=\alpha |a|^2$. Thus, due to the scaling property $\mc{N}(cX)=c^2 \mc{N}(X)$, we get
\begin{align*}
	\mc{N}\left(\frac{|a|}{\sqrt{n}} \sum_{i=1}^n X_i \right) & = \mc{N}\left(\frac{|a|}{\sqrt{n}} \textrm{Gamma}(n\alpha)  \right) \\
	& = \mc{N}\left(\frac{\sqrt{P_N}}{\sqrt{n \alpha}} \textrm{Gamma}(n\alpha)  \right) \\ & = P_N \mc{N}\left(\frac{\textrm{Gamma}(n\alpha)}{\sqrt{n \alpha}}   \right) = P_N  \tau(n \alpha)^{-1}.
\end{align*}
To finish the proof it suffices to use the Shannon bound \eqref{eq:shannon}. 
\end{proof}

\begin{rem}
Since $\mc{N}(N) \leq \var(N)=P_N$ one always has $C_P(N) \geq \frac12 \ln(1+\frac{P}{P_N})$ with equality for $N$ being Gaussian noise. In the proof of Corollary \ref{cor:capacity} we have used a sharper bound on $\mc{N}(N)$, and thus our result improves upon this trivial bound. It follows that $\tau(x) \geq 1$ for $x \geq 1$. In fact numerical simulations show that $\tau$ is a decreasing function of $x$ on $(0,\infty)$ with limit $1$ as $x \to \infty$. 
\end{rem}

\section{Open Problems}\label{sec:op}

In this section we present some open questions related to our study.

\begin{question}\label{que:1}
In this article we considered only nonnegative quadratic forms. It is natural to ask about an analog of Theorem \ref{thm:main} for general quadratic forms. This corresponds to proving an analog of Theorem \ref{prop:diagonal-chaos}, that is, maximizing $h\left(\sum_{i=1}^n d_i g_i^2\right)$ under the constraint $\sum_{i=1}^n d_i^2=1$, where $g_i$ are independent $\mc{N}(0,1)$ random variables. Numerical simulations show that $d_1=1/\sqrt{2}$ and $d_2=-1/\sqrt{2}$ give the maximum for $n=2$, which suggests that for general $n=2k$ a natural candidate for the maximizer would be $d_1=\ldots=d_n = -d_{n+1} = \ldots = -d_{2n}= 1/\sqrt{2n}$. We do not have any predictions in the odd case.
\end{question} 

\begin{question}
Similarly to Question \ref{que:1}, we can ask about the maximum value of $h(\sum_{i=1}^n d_i X_i)$ under the constraint $\sum_{i=1}^n d_i^2 =1$, where $X_i$ are independent standard exponential random variables, that is, random variables with densities $e^{-x}\1_{[0,\infty)}(x)$. Probably the maximizers are be the same as in Question \ref{que:1}.
\end{question} 
 
\begin{question}[See Question 9 in \cite{ENT18}]\label{que:uniform}

Suppose $U_1, \ldots, U_n$ are independent random variables distributed uniformly in $[-1,1]$. Is it true that $h(\sum_{i=1}^n d_i U_i) \leq h(n^{-1/2} \sum_{i=1}^n U_i)$, whenever $\sum_{i=1}^n d_i^2=1$? 
\end{question} 

\begin{question}
More ambitiously, we can ask similar questions about some larger classes of random variables. For example, suppose that $X_1, \ldots, X_n$ are i.i.d. symmetric log-concave random variables (that is, random variables with densities of the form $e^{-V}$, where $V:\R \to (-\infty,\infty]$ is convex). Is it true that $h(\sum_{i=1}^n d_i X_i)$ is maximized when the coefficients $d_i$ are all equal? The answer to this question is not known even for $n=2$. 
\end{question}

\begin{question}[See Question 12 in \cite{ENT18}]
Among all random variables with a fixed variance the one maximizing the entropy is a Gaussian random variable. Suppose $X_1, X_2$ are i.i.d. and suppose that $G$ is a Gaussian random variable satisfying $\var(X_1)=\var(G)$. Is it always true that $h(X_1+X_2) \leq h(X_1+G)$? Let us mention that this inequality does not hold if, instead of assuming that $X_1, X_2$ are i.i.d., we only assume that they are independent with the same variances.
\end{question}

\begin{question}\label{que:pth-for-gauss}
Let $X_i=g_i^2-1$, where $g_i$ are i.i.d. $\mc{N}(0,1)$ random variables, and let $p>2$. What is the maximum value of $\E\left| \sum_{i=1}^n a_i X_i \right|^p$ under $\sum_{i=1}^n a_i^2 =1$? In other words, what is the maximum possible $p$th central moment of a Gaussian quadratic form under a fixed variance? 
\end{question}

\begin{question}\label{que:renyi-entropy}
For a real random variable $X$ with density $f$ we define its R\'enyi entropy of order $\alpha \in (0,\infty) \setminus \{1\}$ via the expression
\[
	h_\alpha(X) = h_\alpha(f) = \frac{1}{1-\alpha} \ln \left( \int f^\alpha \right).
\]
For $\alpha=1$ we can define $h_1 = \lim_{\alpha \to 0^+} h_\alpha(f)$, which recovers the usual entropy, that is $h_1=h$. For $\alpha=\infty$ another limiting procedure justifies the definition $h_\infty(f) = - \ln \|f\|_\infty$, where $\|\cdot \|_\infty$ denotes the essential supremum of $f$ (we shall also use the notation $M(X)=\|f\|_\infty$). For fixed $\alpha \in (0,\infty]$, what is the maximal/minimal possible R\'enyi entropy of order $\alpha$ for a Gaussian quadratic form of  fixed variance? Equivalently, what is the maximum/minimum of $h_\alpha(\sum_{i=1}^n d_i g_i^2)$ under the constraint $\sum_{i=1}^n d_i^2 =1$? We note that the case $\alpha=\infty$ gives bounds on the so-called concentration function $Q(X; \lambda) = \sup_x \mb{P}(x \leq X \leq x+\lambda)$ for Gaussian quadratic forms, which would be of independent interest in probability theory. The same questions can be asked when $g_i^2$ are replaced with arbitrary i.i.d. gamma random variables $X_i$.
\end{question}

\section{Further motivation and discussion} \label{sec:diss}

\subsection{Relation to convex order} Let $X$ and $Y$ be two real random variables. We say that $X$ is smaller than $Y$ in the \textit{convex order} (denoted $X \prec Y$) if for every convex function $\phi:\R \to \R$ one has $\E \phi(X) \leq \E \phi(Y)$. Marshall and Proschan  in \cite{MP65} observed that if the distribution of the vector (with not necessarily independent components) $(X_1,\ldots, X_n)$ has distribution that is invariant under permuting coordinates, then for any convex function $\Phi:\R^n \to \R$ the function $\Psi(a_1,\ldots, a_n) = \E \Phi(a_1X_1,\ldots, a_n X_n)$ is Schur convex, namely if $a,b \in \R^n$ are such that $a \prec b$, then $\Psi(a)\leq \Psi(b)$. To see this recall (see \cite{MOA11}) that $a \prec b$ if and only if there exist nonnegative numbers $\lambda_\pi$ summing up to $1$, such that $a = \sum_{\pi} \lambda_\pi b_\pi$, where $b_{\pi}= (b_{\pi(1)},\ldots, b_{\pi(n)})$ and $\pi$ is a permutation of $\{1,\ldots, n\}$. Observe that $\Psi$ is convex (as an average of convex functions) and permutation symmetric. We thus have
\[
	\Psi(a) = \Psi\left(\sum_\pi \lambda_\pi b_\pi \right) \leq  \sum_\pi \lambda_\pi \Psi\left( b_\pi \right) = \Psi(b).
\]     
In particular, if $\phi:\R \to \R$ is convex and we define $X_a = \sum_{i=1}^n a_i X_i$, then $a \prec b$ implies $\E \phi(X_a) \leq \E \phi(X_b)$.
Indeed, it suffices to consider $\Phi(x_1,\ldots,x_n)= \phi(x_1+\ldots+x_n)$. We conclude that if $a \prec b$, then $X_a \prec X_b$. 

Using the above theory Yu in \cite{Y08} showed that if $X_1,\ldots, X_n$ are i.i.d. log-concave random variables, then for every $a,b \in \R^n$ with $a \prec b$ one has $h(X_a) \leq h(X_b)$. 
In particular, if $\sum_{i=1}^n a_i=1$ and $a_i \geq 0$ for $i=1,\ldots, n$, then 
\[
h\left( \frac1n \sum_{i=1}^n   X_i \right) \leq   h\left(\sum_{i=1}^n a_i X_i\right) \leq h\left(X_1\right).
\] 
To see this it is enough to recall that sums of independent log-concave random variables are log-concave and use the following general observation: if $X \prec Y$ and $Y$ is log-concave, then  $h(X) \leq h(Y)$. Indeed, if $f_X$ and $f_Y$ are the densities of $X$ and $Y$, then by Lemma \ref{lem:entropy-standard-bound} one has
\begin{align*}
	h(X) & = - \int f_X \ln f_X \leq - \int f_X \ln f_Y  = \E [-\ln f_Y(X) ] \\ &  \leq \E [-\ln f_Y(Y) ] = h(Y),   
\end{align*}
where the last inequality follows from the fact that $-\ln f_Y$ is convex, as $Y$ is log-concave.

The conclusion of the above considerations is that comparing functionals $a \to \E \phi(X_a)$ with $\phi$ convex or $a \to h(X_a)$ (under additional assumption of log-concavity) is an easy task when mean is fixed  (note that $a \prec b$ corresponds to $\E X_a = \E X_b$). For example, if $a \prec b$ then we always have $\E |X_a|^p \leq \E |X_b|^p$ for $p \geq 1$ which, in particular, implies that $\var(X_a) \leq \var(X_b)$. However, Problem \ref{prob:general} is much more delicate, since in this problem, instead of fixing the mean, we fix the variance. In other words, instead of fixing $\sum_{i=1}^n a_i$, we fix $\sum_{i=1}^n a_i^2$. It seems that with this constraint no general statements concerning Schur comparison can be made. For example, in \cite{MNT20} it was shown that even if $X_1,\ldots, X_n$ are i.i.d. symmetric log-concave real random variables, then $a^2 \prec b^2$ (here $a^2=(a_1^2,\ldots, a_n^2)$) does not imply $h(X_a) \geq h(X_b)$, which would be a natural conjecture based on the entropy monotonicity in the Central Limit Theorem. 

However, due to the result from \cite{ENT18}, if $X_1,\ldots, X_n$ are i.i.d. Gaussian mixtures, then $a^2 \prec b^2$ implies $h(X_a) \geq h(X_b)$. Let us recall how convex ordering is applied in this case. Since $X_i$ are Gaussian mixtures, there exist i.i.d. positive random variables $R_i$ and independent $\mc{N}(0,1)$ random variables such that $X_i \sim R_i G_i$. Thus, 
\[
X_a = \sum_{i=1}^n a_i X_i \sim \sum_{i=1}^n a_i R_i G_i \sim \left(\sum_{i=1}^n a_i^2 R_i^2 \right)^{1/2} G_1
\]
and, in particular, $X_a$ is itself a Gaussian mixture.
This is how squares of $a_i$ are introduced in the proof. Let $f_a$ be the density of $X_a$. If we now apply Lemma \ref{lem:entropy-standard-bound}, we see that it is enough to show the inequality $\E[-\ln f_a(X_b)] \leq \E[-\ln f_a(X_a)]$. This can be rewritten as 
\[
\E \phi\left(\sum_{i=1}^n b_i^2 R_i^2 G_1^2 \right) \leq \E \phi\left(\sum_{i=1}^n a_i^2 R_i^2 G_1^2 \right), 
\]
where $\phi(x) = -\ln f_a(\sqrt{x})$.  Since $f_a$ is a density of a Gaussian mixture, it is of the form $f_a(x) = \int_0^\infty e^{-x^2/2t^2} \dd \mu(t)$ for some positive measure $\mu$, and thus $f_a(\sqrt{x}) =\int_0^\infty e^{-x/2t^2} \dd \mu(t)$. This function is clearly log-convex (just apply H\"older's inequality). Thus, $\phi$ is concave and the result follows from the fact that $a^2 \prec b^2$ implies $\sum_{i=1}^n a_i^2 R_i^2 G_1^2  \prec \sum_{i=1}^n b_i^2 R_i^2 G_1^2$ mentioned above (note that here we only need the property that $G_1^2(R_1^2, \ldots, R_n^2)$ has distribution invariant under permutation of coordinates).     

Finally, we mention a very general fact about the comparison of distribution functions of weighted sums of i.i.d. log-concave symmetric random variables due to Proschan, see \cite{P65}: if $X_i$ are i.i.d. symmetric log-concave random variables, then $\mb{P}(X_a \geq t)$ is Schur-convex in $a$, for any fixed $t>0$.    

\subsection{Unique crossing theorem} Consider $X_a = \sum_{i=1}^n a_i X_i$, where $X_i$ are i.i.d. Gamma($\alpha$) random variables.
In \cite{BDHP87} the authors proved Schur-convexity of tails and distribution functions of $X_a$, as functions of $a$, on certain half-lines and intervals, see also \cite{B96} for bounds in the special case of Gaussian quadratic forms. In \cite{DP90} Diaconis and Perlman conjectured that the distribution functions of $X_a$ and $X_b$ cross exactly once when $a\prec b$. This conjecture has been verified by Yu in \cite{Y17} for gamma distributions of shape parameters $\alpha \geq 1$ and disproved for $\alpha<1$. See also \cite{RS15} for some partial results.

What is also worth mentioning is that Sz\'ekely and Bakirov in \cite{SB03} determined the quantities $I_n(x) = \inf_a  \mb{P}(\sum_{i=1}^n a_i g_i^2$ $\leq x)$, where the infimum is taken under the constraint $\sum_{i=1}^n a_i=1$ and $a_i \geq 0$.       
 
\subsection{Khinchine inequalities} Problem \ref{prob:general} is often considered together with its moment counterpart, namely the problem of maximizing the $p$th moment $(\E|S|^p)^{1/p}$ of $S=\sum_{i=1}^n a_i X_i$, for $p>2$, where $X_i$ are i.i.d. symmetric random variables, under a fixed variance. The latter is equivalent to proving the Khintchine-type inequality $(\E|S|^p)^{1/p} \leq C (\E|S|^2)^{1/2}$ with the optimal constant $C$ (depending on $n$ and on the distribution of $X_1$). This has been studied extensively, starting from the case of $X_i$ being symmetric Bernoulli random variables, see the works of Ste\v{c}kin \cite{S61}, Whittle \cite{W60}, and Haagerup \cite{H82}. Later, Latała and Oleszkiewicz solved this problem for random variables distributed uniformly in $[-1,1]$, see \cite{LO}. Averkamp and Houdr\'e in \cite{AH03} settled the case of $X_i$ being  Gaussian mixtures, which was further used in \cite{ENT18} to study the case of $X_i$ being (non-independent) coordinates of a random vector uniformly distributed on $B_q^n=\{(x_1,\ldots, x_n) \in \R^n: \sum_{i=1}^n |x_i|^q \leq 1\}$, for $q \in (0,2]$. The case $q>2$ was treated in \cite{ENT18-2}.  We mention that usually Problem \ref{prob:general} is much harder than its moment-counterpart, the reason being that the $p$th moment is a linear function of the underlying distribution, whereas the entropy is non-linear. For example, the entropy counterpart of the result of Latała and Oleszkiewicz for $X_i$ distributed uniformly in $[-1,1]$ is not known, see Question \ref{que:uniform} in Section \ref{sec:op}. However, surprisingly, the present article solves Problem \ref{prob:general} for $X_i=g_i^2-1$, in which case the moment analog is still open, see  Question \ref{que:pth-for-gauss}. 

The problem of estimating the $p$-th moment of a Gaussian quadratic form and, more generally, of a Gaussian chaos of arbitrary order, has been extensively studied. In particular, it is known that for an arbitrary $n \times n$  matrix $A$ one has, up to absolute constants,  
\[
	\left(\E \left| \scal{A G_n}{G_n} - \E \scal{A G_n}{G_n} \right|^p \right)^{1/p} \sim \sqrt{p} \| A\|_{HS} + p \|A\|,
\]    
where $\|A\|_{HS}$ stands for the Hilbert-Schmidt norm and $\|A\|$ denotes the operator norm. For this result and its extensions to Gaussian chaoses of higher degree see the works \cite{HW71, L06, AW15} of Hanson and Wright, Latała, and Adamczak and Wolff. Unfortunately, once we are not allowed to lose any constants, as in Question \ref{que:pth-for-gauss}, the techniques that lead to the above results cannot easily be adapted. 

\subsection{Hadwiger \& Ball's cube slicing inequalities} In \cite{H72} Hadwiger proved that $|[-\frac12, \frac12]^n \cap a^{\perp}| \geq 1$ for any $a \ne 0$, that is, the $(n-1)$-dimensional central section of the cube has volume at least $1$, with equality for $a=(1,0,\ldots,0)$. Later in his celebrated work \cite{B86-ball} Ball showed that for any $a \ne 0$ one has $|[-\frac12, \frac12]^n \cap a^{\perp}| \leq \sqrt{2}$. One can rephrase these results as follows: if $U_1,\ldots, U_n$ are i.i.d. random variables distributed uniformly in $[-\frac12,\frac12]$ and if $f_a$ denotes the density of $\sum_{i=1}^n a_i U_i$, then  $1 \leq f_a(0)\leq \sqrt{2}$ for $a \in S^{n-1}$. Indeed, one has $f_a(0)=|[-\frac12, \frac12]^n \cap a^{\perp}|$.

The quantity $f_a(0)$ is equal to $\max_{x \in \R} f_a(x)$, due to the log-concavity and symmetry of $f_a$. As $h_\infty(f) = -\ln \|f\|_\infty$, the above inequalities can be rewritten as $-\frac12\ln 2 \leq h_\infty(\sum_{i=1}^n a_i U_i) \leq 0$. We can now see that in Question \ref{que:1} we ask for an analogue of the upper bound for the $h$ functional.          

\subsection{Concentration function of Gaussian quadratic forms}

In the present article we considered entropy of $\sum_{i=1}^n d_i g_i^2$, where $g_i$ are independent standard Gaussian random variables. The same object has recently been considered by Bobkov, Naumov and Ulyanov in \cite{BNU20}, where estimates for its $h_\infty$ functional have been obtained. The authors showed that for positive $d_i$ with $\sum_{i=1}^n d_i^2 =1$ the following bounds on the maximum of the density of $X$ hold
\[
	\frac{c_0}{\sqrt[4]{1-\max_{i} d_i^2}} \leq M\left( \sum_{i=1}^n d_i g_i^2 \right) \leq \frac{c_1}{\sqrt[4]{1-\max_{i} d_i^2}},
\]  
where $c_0>0.013$ and $c_1<1.129$. The difference between this result and bounds discussed in Question \ref{que:pth-for-gauss} is that in the above estimates the dependence on $d_i$ is taken into account (whereas the bounds are tight only up to a universal constant) and in Question \ref{que:renyi-entropy} we ask for  sharp bounds independent of $d_i$, under our usual constraint $\sum_{i=1}^n d_i^2=1$.    

Bounds on $M(\sum_{i=1}^n a_i X_i)$ for general independent random variables $X_i$ are also known. Let us mention here the result of Bobkov and Chirstryakov from \cite{BC14}: if $X_i$ are independent real random variables with  $M(X_i)$ finite, then for all real numbers $a_1, \ldots, a_n$ with $\sum_{i=1}^n a_i^2=1$ one has
\[
	\frac{1}{M^2(\sum_{i=1}^n a_i X_i)} \geq \frac12 \sum_{i=1}^n \frac{1}{M^2(X_i)}.
\] 
Here the constant $\frac12$ is best possible. The proof of this fact combines Balls slicing inequality from \cite{B86-ball} with a result of Rogozin from \cite{R87}.

Let us also mention that development of estimates for concentration functions of sums of independent random variables dates back to the works of Lévy and Kolmogorov, see \cite{L37,K60}. See also the works of Rogozin \cite{R61-1,R61-2}, Kesten \cite{K69}, and Esseen \cite{E68}. 

\subsection{Sums of gamma distributions in applications.}
Sums of independent gamma random variables arise in applied contexts in statistics, actuarial science and engineering. For example, Gaussian quadratic forms occur as the limiting distributions of degree two degenerate $U$-statistics, see \cite{G77,S80,ACF82}, and as limiting distributions of the $\chi^2$ goodness-of-fit statistics, see \cite{CL54, M78}. They also show up naturally in the context of estimating the trace of an $n \times n$ symmetric positive semi-definite matrix $A$ by using the so-called Gaussian estimator $\tr_N(A) = N^{-1}\sum_{i=1}^N G_i^T A G_i$, where $G_i$ are i.i.d. $\mc{N}(0,I_n)$ random vectors, see \cite{RS15}.  

Sums of exponential random variables occur in the form $-\ln(\prod_{i=1}^n U_i^{a_i})$, where $U_i$ are uniform on $(0,1]$ and the quantity $\prod_{i=1}^n U_i^{a_i}$ is a weighted Fisher statistic for combining independent $p$-values $U_1,\ldots, U_n$, see \cite{G55}. They also arise as first-passage-time distributions in certain birth-and-death processes, see \cite{F71,BS87}.  

Sums of independent gamma random variables are used in queuing theory and storage models, see \cite{Pr65}, as well as in the risk of portfolio theory, see \cite{H01}.   

\vspace{0.3cm}

\section{More on the method of intersecting densities}\label{sec:method}

\subsection{General framework.} Suppose $A \subseteq \R$ is connected and $\mc{F}$ is a certain class of functions. Suppose we consider functionals  of the form $\Phi(f) = \int_A gf$ for some function $g$ (here we assume integrability of $fg$ for all $f \in \mc{F}$). Let $g, g_1, g_2, \ldots, g_n$ be certain functions and let $\Phi, \Phi_1,\ldots, \Phi_n$ be corresponding functionals. Suppose our goal is to maximize $\Phi(f)$ under constraints $\Phi_i(f)=m_i$ for all $i=1,\ldots, n$. In other words, our goal is to find the quantity
\[
	M_{\mc{F}}(m_1,\ldots, m_n) = \sup \left\{\Phi(f): \ f \in \mc{F}, \ \Phi_i(f)=m_i, \ i=1, \ldots, n\right\}.
\]
This clearly is a fundamental optimization problem arising in many different contexts. If $\mc{F}$ is a non-linear spaces of function (such as the space of log-concave functions), it is usually hopeless to deal with more than one or two constraints. The method of intersecting densities is a way to overcome these difficulties in certain situations. 

Let us now describe a general framework of our  method. Suppose  that we have two functions $f,f_0$ and we would like that the inequality $\Phi(f) \leq \Phi(f_0)$ holds under constraint of the aforementioned form (for example, $f_0$ is our candidate for the maximizer in the above optimization problem).  Our inequality can be written as $\int_A g(f_0-f) \geq 0$. Usually constraints prevent us from having pointwise estimate $g(f_0-f) \geq 0$, so one has to rewrite the inequality in some way. And here comes our crucial observation: due to the constraints the inequality can equivalently be written as
\[
 	\int_A (f_0-f)\left(g- \sum_{i=1}^n a_i g_i \right) \geq 0,  
\]      
where $a_i$ are arbitrary. Now the idea is to explore our freedom of the choice of $a_i$. Suppose now that the following  conditions hold:

\begin{itemize}
\item[(I)] The function $f_0-f$ changes sign exactly $n$ times.
\item[(II)] The matrix $(g_i(x_j))_{i,j=1}^n$ is invertible for all $x_1<\ldots <x_n$. 
\item[(IIIa)] For any choice of $a_i$ the function $h=g- \sum_{i=1}^n a_i g_i$ changes sign at most $n$ times.
\item[(IIIb)] If $h$ has exactly $n$ zeros, then these zeros are sign change points. 
\end{itemize}

\noindent Then we can make the integrand of constant sign in the following way: take points $x_1<\ldots< x_n$ where $f_0-f$ changes its sign and choose $a_i$ such that $h(x_i)=0$, using solvability of the corresponding system of linear equations guaranteed by (II). Then due to (IIIb) the function $h$ changes its sign in the points $x_i$, where $f_0-f$ changes its sign, and nowhere else according to (IIIa). Thus, the integrand has a constant sign (if it is negative, then $f_0$ is the minimizer, not the maximizer). The sign of the integrand is usually easy to determine by checking it in some concrete point, or in the limit as the argument converges to infinity.     

Let us now focus on one of the following two special cases of moment-type constraints: 
\begin{itemize}
\item[(A)] $A=[0,\infty)$ and $g_i(t)=t^{p_i}$ for some $p_i \in \R$,
\item[(B)] $A =[L,\infty)$ and $g_i(t) =t^{n_i}$, where $n_i$ are non-negative integers and $L \in[-\infty,\infty)$.
\end{itemize}
In both cases (II) is automatically fulfilled. Indeed in case (B) we get the usual Vandermonde determinant whereas for (A) one can use Lemma 22 from \cite{ENT18-2}. Verifying conditions (I) and (III) might not be an easy task and may lead to various issues.

\subsection{Previous development}
In \cite{ENT18-2} in Chapter 4 the technique was used to solve the log-concave moment problem, namely to find log-concave non-increasing probability densities on $[0,\infty)$ maximizing and minimizing the integral $\int_0^\infty t^p f(t) \dd t$ subject to constraints $\int_0^\infty t^{p_i} f(t) \dd t = m_i$, $i=1,\ldots, n$, where $p_1<p_2<\ldots<p_n$. Here $g(t)=t^p$, in which case verifying (III) is easy (see Lemma 19 in \cite{ENT18-2}). Checking (I) is also not a big issue once good candidates $f_0$ for extremizers are found. Thus, the difficulty of this result lies in the conceptual framework related to the inductive scheme rather than in technical issues.  

In Section 3 of \cite{ENT18-2} a simple proof of the following result of Lata{\l}a and Oleszkiewicz from \cite{LO} was given using the technique of intersecting densities: if $p \geq 2$ and $U_i$ are uniform on $[-1,1]$, then $\E|\sum_{i=1}^n a_i U_i|^p$ is a Schur concave function of $(a_1^2,\ldots, a_n^2)$, whereas for $1 \leq p \leq 2$ it is Schur convex. Please note that this result seems very different from the previous one and still the same method can be applied. The proof relies on the following fact (an analogue of Lemma \ref{lem:wyrownywanie-wystarczy} in the present paper): $\E|X_\lambda+Y|^p \leq \E|X_{\lambda'}+Y|^p$ for $0<\lambda<\lambda'<\frac12$ with $X_\lambda = \sqrt{\lambda}U_1+\sqrt{1-\lambda} U_2$, where $Y$ is any unimodal density. The function $g(x)=\E_Y|\sqrt{x}+Y|^p$ turns out to be convex (in fact, as $Y$ unimodal, it is enough to check this for $Y$ being uniform on $[-1,1]$, in which case it is a simple computation). Let $f_\lambda$ be the density of $X_\lambda$. Thus, we want to show that
\[
	\int_0^\infty g(x^2)(f_{\lambda'}(x)-f_{\lambda}(x)) \dd x \geq 0.
\]
This can be rewritten as 
\[
	\int_0^\infty (g(x^2)-(a_1+a_2 x^2))(f_{\lambda'}(x)-f_{\lambda}(x)) \dd x \geq 0.
\]
Now assumption (I) is straightforward to verify and (III) follows from the fact that $g(t)=a_1+a_2t$ has at most two solution due to convexity of $g$. We can now see that this reasoning has a similar structure to the one presented in this article, but the details are different: in the present paper we use the method with $n=3$ constraints instead of just two and the verification of (I) is much more complicated, see Lemma \ref{lem:3-times} (here $U$ plays the role of $X_\lambda$ and $V$ plays the role of $X_{\lambda'}$). Note that also (III) holds  for quite a different reason (in our proof we do not rely on convexity, but on the fact that an equation of the form $\ln x = a_2x^2+a_1x+a_0$ can have at most three solutions).

\subsection{Challenges and obstacles} 
Let us now present two examples, where the verification of assumptions (I) and (III) is not an easy task. This shows that while our method is very general and can be applied to many different problems, the details are usually quite different and it is hard to believe that all these cases can be unified.

\subsubsection{The most Gaussian direction in the cube (Question \ref{que:uniform})} 
Suppose we want to show that the maximum of $h(\sum_{i=1}^n a_i U_i)$ for $U_i$ uniform on $[-1,1]$, under the constraint $\sum_{i=1}^n a_i^2=1$, is given by $a_i=n^{-1/2}$. Then we can follow the strategy of the present paper (see Lemma \ref{lem:entropy-standard-bound}) and then proceed as in the proof of the result of Lata{\l}a and Oleszkiewicz \cite{L06}, in order to show that $\E \Phi_n(\sum_{i=1}^n a_i U_i) \leq \E \Phi_n(n^{-1/2}\sum_{i=1}^n U_i)$, where $\Phi_n= -\ln p_n$, with $p_n$ being the density of $n^{-1/2} \sum_{i=1}^n U_i$, that is, an affine image of the Irwin-Hall distribution. Here verifying (I) is the same as in the proof for moments. However, (III) is now a difficult technical problem. One can check that it would be enough to show that $(\ln p_n(x))'''\leq 0$ for $x>0$. Numerical simulations show that this is indeed true for $n \geq 7$.

\subsubsection{Khinchine inequalities for $c_p e^{-|x|^p}$ densities} 
Suppose $Y_1,\ldots, Y_n$ are i.i.d. random variables with densities of the form $c_pe^{-|x|^p}$ for $p \geq 2$. In this case an analogue of the result of Lata{\l}a and Oleszkiewicz should hold. In fact, checking condition (III) is precisely the same as in the above proof. The only  problem is to verify (I), which is  technically difficult, since the density $f_\lambda$ of $X_\lambda = \sqrt{\lambda} Y_1 + \sqrt{1-\lambda} Y_2$ is now given by a complicated expression and finding the number of sign changes of $f_\lambda - f_{\lambda'}$ turns out to be challenging.

\vspace{0.5cm} 

\paragraph{\textbf{Acknowledgments.}} We would like to thank Tomasz Tkocz and Alexandros Eskenazis for stimulating discussions, which helped us formulate Theorem \ref{prob:general} in its present general form. We are also grateful to the anonymous referees for useful comments.

\end{document}